\documentclass[12pt]{article}
\usepackage{amsmath,amsfonts,amssymb,amsthm}
\usepackage{tcolorbox}
\definecolor{pantone312}{HTML}{009DD1}
\usepackage{mathtools}
\usepackage{bbm}
\usepackage{graphicx}
\usepackage{tabularx}
\usepackage{color}
\definecolor{darkgreen}{rgb}{0,0.55,0}
\newcommand{\grad}{\nabla}

\newcommand{\laplace}{\Delta}

\renewcommand{\div}{\grad\cdot}
\newcommand{\curl}{\grad\times}

\newcommand{\N}{\mathbbm{N}}

\newcommand{\R}{\mathbbm{R}}

\newcommand{\T}{\mathbbm{T}}

\usepackage{sectsty}

\sectionfont{\large}
\subsectionfont{\normalsize}
\subsubsectionfont{\normalsize}
\paragraphfont{\normalsize}

\newcommand{\on}{\omega^{\nu}}
\newcommand{\un}{u^{\nu}}
\newcommand{\ok}{\omega^{k}}
\newcommand{\uk}{u^{k}}

\newcommand{\eps}{\varepsilon}

\newcommand{\dd}{{\rm d}}
\newcommand{\dx}{\,\dd x}
\newcommand{\dt}{\, \dd t}

\def\XXint#1#2#3{{\setbox0=\hbox{$#1{#2#3}{\int}$ }
\vcenter{\hbox{$#2#3$ }}\kern-.59\wd0}}

\newtheorem{prop}{Proposition}
\newtheorem{theorem}{Theorem}
\newtheorem{lemma}{Lemma}

\pagecolor{white}

\begin{document}
\phantom{ }
\vspace{4em}

\begin{flushleft}
{\large \bf On the vanishing viscosity limit for 2D incompressible flows with unbounded vorticity}\\[2em]
{\normalsize \bf Helena J. Nussenzveig Lopes}\\[0.5em]
\small Instituto de Matem\'atica,  Universidade Federal do Rio de Janeiro, Brazil.\\
E-mail: hlopes@im.ufrj.br\\[3em]
{\normalsize \bf Christian Seis}\\[0.5em]
\small Institut f\"ur Analysis und Numerik,  Westf\"alische Wilhelms-Universit\"at M\"unster, Germany.\\
E-mail: seis@wwu.de\\[3em]
{\normalsize \bf Emil Wiedemann}\\[0.5em]
\small Institut f\"ur Angewandte Analysis, Universit\"at Ulm, Germany.\\
E-mail: emil.wiedemann@uni-ulm.de\\[3em]

{\bf Abstract:} We show strong convergence of the vorticities in the vanishing viscosity limit for the incompressible Navier-Stokes equations on the two-dimensional torus, assuming only that the initial vorticity of the limiting Euler equations is in $L^p$ for some $p>1$. This substantially extends a recent result of Constantin, Drivas and Elgindi, who proved strong convergence in the case $p=\infty$. Our proof, which relies on the classical renormalization theory of DiPerna-Lions, is surprisingly simple. 
\end{flushleft}

\vspace{2em}

\section{Introduction}

The vorticity formulation of the two-dimensional Navier--Stokes equations describing the motion of an incompressible viscous fluid with planar symmetry is given  by the nonlinear advection-diffusion equation
\begin{equation}\label{1}
\partial_t \on + \un\cdot \grad \on = \nu \laplace \on + g^{\nu},
\end{equation}
where  $\on$ descibes the vorticity of the fluid, $\un$ is the corresponding velocity field,  $g^{\nu}$ represents a forcing term (which is given, but will later be chosen depending on $\nu$), and $\nu$ is the positive viscosity constant.   In order to avoid boundary effects, we consider the evolution on the torus $\T^2$.  It is well-known that $\un$  can  be reconstructed from $\on$ with the help of the Biot--Savart law
\begin{equation}
\label{6}
\un = -\grad^{\perp}(-\laplace)^{-1} \on,
\end{equation}
where $z^{\perp} = (-z_2,z_1)$ is the rotation of the point $z=(z_1,z_2) $ and $(-\laplace)^{-1}$ denotes the inverse Laplacian. The initial condition will be denoted by $\on_0$, that is,
\begin{equation}\label{2}
\on(0) = \on_0.
\end{equation}
In order to make sense of \eqref{6}, it is necessary to require that the vorticity field has zero mean on the torus, which is preserved by the Navier--Stokes equations as long as the forcing term $g^{\nu}$ has zero mean, which we will assume from here on. Notice that assuming that both $\omega^{\nu}$ and $g^{\nu}$ have zero mean is modelling-wise very natural in the present setting because both quantities  are obtained as curls from the fluid's velocity and body forcing, respectively, and thus, for instance, $\int_{\T^2}\omega^\nu\dx = \int_{\T^2}\curl \un\dx=0$ as a consequence of the periodic boundary conditions on the torus.

We will allow both initial vorticity and forcing term to be irregular and unbounded, and dependent on $\nu$. More precisely, for some fixed time $T>0$ and exponent $p\in (1,\infty)$, we suppose that
\begin{equation}\label{3}
\on_0 \in L^p(\T^2) ,\quad g^\nu\in L^1((0,T);L^p(\T^2)).
\end{equation}
In this setting, there exists a unique (mild) solution $\on\in L^{\infty}((0,T);L^p(\T^2))$ satisfying the energy estimate
\begin{equation}\label{4}
\|\on\|_{L^{\infty}_tL^p_x} + \nu \|\grad |\omega^\nu|^{\frac{p}2}\|^{\frac2p}_{L_t^2 L_x^2}\lesssim \|\on_0\|_{L^p_x} + \|g^\nu\|_{L^1_tL^p_x}.
\end{equation}
Here and in the following, we use the short notation $L^q_tL^p_x$ for $L^q((0,T);L^p(\T^2))$ and analogous intuitive notations for conceptually similar function spaces.

Our main concern in this paper is the small viscosity behavior. In the singular limit when $\nu\to 0$, assuming the initial data and forcing term converge appropriately, the Navier--Stokes vorticity equations \eqref{1}, \eqref{6} reduce  to the Euler vorticity equations
\begin{align}
\partial_t \omega + u\cdot \grad \omega &= g,\label{5}\\
 u & = -\grad^{\perp}(-\laplace)^{-1}\ast \omega,\label{7}
\end{align}
describing thus the evolution of an incompressible inviscid fluid in terms of the vorticity field $\omega$.  As before, $u$ is the associated velocity field, and the advection equation \eqref{5} is equipped with the initial condition 
\begin{equation}\label{21}
\omega(0) = \omega_0,
\end{equation}
which is obtained from the viscous vorticity as the inviscid limit
\[
\omega^{\nu}_0\to \omega_0\quad\mbox{strongly in }L^p_x.
\]
Likewise, we assume the convergence of the forcing terms,
\[
g^{\nu}\to g\quad\mbox{strongly in }L^1_tL^p_x.
\]
Thus, the Cauchy problem \eqref{5}, \eqref{7}, \eqref{21} has to be solved under the assumption 
\[
\omega_0 \in L^p(\T^2),\quad g\in L^1((0,T);L^p(\T^2)).
\]

We remark that outside the Yudovich class \cite{Yudovich63}, that is, for $p<\infty$, it is a long-standing open problem whether solutions to the Euler vorticity equations are unique in   $L^{\infty}((0,T);L^p(\T^2))$. On the positive side, vanishing viscosity solutions are known to exist and it has been proved in \cite{FilhoMazzucatoNussenzveig06,CrippaSpirito15,CrippaNobiliSeisSpirito17} that they can be classified as renormalized solutions in the sense of DiPerna and Lions \cite{DiPernaLions89} for any $p\in [1,\infty]$. In fact, in these papers, the unforced equations, i.e., \eqref{1}  and \eqref{5}  with $g^{\nu}=g=0$, are studied, and for $p<2$ the additional assumption is made that the initial vorticity is smooth. At least as long as $p>1$, the arguments therein are almost identical if the forcing terms are included and the  smoothness assumption on the initial data is dropped. To keep the presentation in the present paper concise, we will only consider integrability exponents $p>1$. We will comment on $p=1$ later.

We recall that a solution to the advection equation \eqref{5} is called \emph{renormalized} if for any $\beta\in C^{\infty}(\R)$ that is bounded, has bounded derivatives and vanishes at the origin, it holds that
\begin{equation}
\label{8}
\partial_t \beta(\omega) + u\cdot \grad \beta (\omega) = \beta'(\omega)g.
 \end{equation}
As this identity is certainly true for smooth solutions, the validity of \eqref{8} can be roughly interpreted as the validity of the chain rule. Notice also that for small $p$, the advection equation \eqref{5} can no longer be interpreted in the sense of distributions. Indeed, if the vorticity field $\omega$ is $L^{\infty}_tL^p_x$ integrable, Calder\'on--Zygmund theory  for the Biot--Savart law \eqref{7} and Sobolev embeddings reveal that the velocity field  $u$ belongs to $L^{\infty}_tL^{ 2p/(2-p)}_x$, and thus, the transport nonlinearity  $u\omega$ is integrable   only if $p\ge \frac43$. For smaller integrability exponents, different interpretations   of the advection equation \eqref{5} are proposed in the literature. The probably most prominent ones use  symmetrization arguments as those introduced by Delort, e.g.,  \cite{Delort91,VecchiWu93, Schochet95, BohunBouchutCrippa16}. More relevant for our purposes is the 
notion of  renormalized solutions in the sense of DiPerna and Lions, given in \eqref{8}. Here, because $\beta(\omega)$ is bounded, the nonlinearity $u\beta(\omega)$ in \eqref{8} is again meaningful, and the equation can thus  interpreted in the sense of distributions.

Notice that an important consequence of the renormalization theory is the fact that the inviscid energy estimate holds true, that is
\begin{equation}\label{400}
\|\omega\|_{L^{\infty}_tL^p_x} \le \|\omega_0\|_{L^p_x} + p\|g\|_{L^1_tL^p_x}.
\end{equation}

The main result in \cite{CrippaSpirito15} now states that the sequence of viscous vorticity fields $\{\omega^{\nu}\}_{\nu>0}$ is precompact in $L^{\infty}_tL^p_x$ endowed with the weak* topology, and the limit object $\omega$ is a solution of the Euler vorticity equation \eqref{5}, \eqref{7} when the advection equation is interpreted in the sense of \eqref{8}.  Our main goal of the present paper is to upgrade this convergence result to \emph{strong} convergence in $L^{\infty}_tL^p_x$. In what follows, we will accordingly  select a subsequence $\{\omega^{\nu_k}\}_{k\in\N}$ with $\nu^k\to 0$ and a vorticity field $\omega$ satisfying \eqref{8}, \eqref{5} such that
\begin{equation}\label{401}
\omega^{\nu_k} \to \omega\quad\mbox{weakly* in }L^{\infty}_tL^p_x.
\end{equation}

With these notations fixed, our main result is the following.

\begin{theorem}
\label{T1}
It holds that
\begin{equation}
\label{9}
 \omega^{\nu_k} \to \omega\quad  \mbox{strongly in }C_tL^p_x.
\end{equation}
\end{theorem}

Even in the well-studied Yudovich class $L^{\infty}_tL^{\infty}_x$, this result has been established only very recently \cite{ConstantinDrivasElgindi19} by using sophisticated arguments that involve borderline regularity estimates for the Biot--Savart kernel, energy estimates for advection-diffusion-type equations with variable integrability exponents, and new uniform short time estimates on vorticity gradients. The present work thus largely extends these results to vorticity fields in $L^{\infty}_tL^p_x$ for the full range of exponents $p\in(1,\infty]$ and moreover substantially simplifies the proofs by reducing the argumentation to a suitable application of the DiPerna--Lions theory developed in \cite{DiPernaLions89}. Our gentle extension of this theory is formulated in Proposition \ref{P1} below. 

We hasten to add that the more sophisticated proof in~\cite{ConstantinDrivasElgindi19} has additional merits concerning convergence rates. Indeed, simple examples indicate that even for vorticity fields  in the Yudovich class $L^{\infty}_tL^{\infty}_x$, no rates of strong $L^p$ convergence can exist. The situation changes if additional regularity assumptions on the initial data are imposed, see Corollary 2 in \cite{ConstantinDrivasElgindi19}. However, rates of \emph{weak} convergence can be obtained, more precisely, bounds on the Wasserstein distance between the Navier--Stokes and the Euler vorticity fields, cf.~\cite{Seis20}. These bounds are an extension of stability estimates for the Euler equations in the Yudovich class developed in \cite{Loeper06a}, and are closely related to estimates on the convergence rates of the associated velocity fields as obtained in \cite{ConstantinWu95,ConstantinWu96,Chemin96,AbidiDanchin04,Masmoudi07}.
We remark that even the related vanishing diffusivity problem in the linear setting with Sobolev vector fields was treated only recently  \cite{Seis17,Seis18}.

{
We like to note that in the unforced situation $g^{\nu} = g=0$, the argument for strong convergence of the vorticities is even shorter as the one presented below, and builds up on  lower semicontinuity for weakly converging sequences, $\|\omega(t)\|_{L^p_x}\le \liminf_{k\to \infty} \|\omega^{\nu_k}(t)\|_{L^p_x}$, dissipation of the viscous solutions, $\|\omega^{\nu_k}(t)\|_{L^p_x} \le \|\omega_0^{\nu_k}\|_{L^p_x}$,  conservation of vorticity in the limit $\|\omega(t)\|_{L^p_x} = \|\omega_0\|_{L^p_x}$, and an Arzel\`a--Ascoli-type argument to guarantee continuity in time. Such a strategy is also indicated in~\cite[Remark 2]{ConstantinDrivasElgindi19}. A modification of the argument still works in the presence of forcing terms for $p=  2$ and then by interpolation and iteration also for larger values of $p$. Our main contribution here is thus the treatment of the setting with smaller integrability exponents.
}

At this point, we like to comment in which regards the $p=1$ case differs from the situation considered here. Most notably, if the vorticity is only $L^1$ integrable in space, the associated velocity field is no longer a Sobolev function. Indeed, the velocity gradient is merely a singular integral of an $L^1$ function for which the standard DiPerna--Lions theory fails. An extension of this theory to such vector fields has been developed in \cite{BouchutCrippa13,CrippaNobiliSeisSpirito17}, at least for the transport equation \eqref{5}. It is, however, not immediate how the concepts developed in these works translate to the diffusive setting \eqref{1}, even though their validity has to be expected due to better regularity properties. 
 A second issue concerns the compactness, which in $L^1$ is slightly more delicate as equi-integrability statements have to be established. This has been achieved in \cite{CrippaNobiliSeisSpirito17} under the additional assumption that the initial viscous vorticity fields are smooth. It seems to us that the extension of the results in \cite{BouchutCrippa13,CrippaNobiliSeisSpirito17} (and also preliminary works on the Euler equations \cite{BohunBouchutCrippa16}) to the present situation, although possible, would require a more detailed discussion. To keep the presentation as short and concise as possible, we thus simply drop the $L^1$ setting.

Shortly before finishing this manuscript, we learned that Ciampa, Crippa and Spirito have independently been preparing a result very similar to ours by exploiting the DiPerna--Lions theory in a similar manner as we do \cite{CiampaCrippaSpirito20}. In this work, the authors do consider the $p=1$ case and derive, in addition, estimates on convergence rates under additional regularity assumptions on the initial vorticity.
 

\section{Proof}

Theorem \ref{T1} essentially follows from the analogous result in the linear setting by a simple mollification argument. The linear result is a small extension of DiPerna and Lions's theory \cite{DiPernaLions89}.

\begin{prop}
\label{P1}
Let $p\geq 1$ and $\tilde p\geq 2$ be such that $\tilde p\geq p'$, where $p'$ is the dual exponent to $p$ (so that $\frac1p+\frac{1}{\tilde p}\geq 1$). Let $\theta^k\in L^{\infty}_tL^{\tilde p}_x$ be the unique distributional (and thus renormalized) solution to the linear advection-diffusion equation 
\begin{equation}\label{15}
\partial_t \theta^k + u^k \cdot \grad \theta^k  = \nu^k \laplace \theta^k + g^k
\end{equation}
with initial datum $\theta^k(0) = \theta_0^k\in L^{\tilde p}$, where $u^k\in L^1_tW^{1,p}_x$ with $\grad\cdot u^k=0$, and $g^k\in L^1_tL^{\tilde p}_x$. Suppose that 
\begin{gather*}
{\nu^k\to0,}\\
\theta^k_0 \to \theta_0 \quad \mbox{strongly in }L^{\tilde p}_x,\\
u^k \to u \quad \mbox{strongly in }L^1_tL^p_x,\\
\grad u^k\to \grad u\quad \mbox{weakly in } L^1_tL^{p}_x,\\
g^k\to g\quad \mbox{strongly in } L^1_tL^{\tilde p}_x.
\end{gather*}
Then, for any $q<{\tilde p}$, 
\[
\theta^k\to \theta\quad \mbox{strongly in }C_tL^q_x,
\]
where $\theta$ is the renormalized solution to the linear advection equation 
\[
\partial_t \theta + u\cdot \grad \theta =g
\]
with initial datum $\theta(0) = \theta_0$.
\end{prop}


\begin{proof}[Proof of Proposition \ref{P1}]
\emph{Step 1. Compactness.} Because $\theta^k$ is a renormalized solution (see~\cite[Section IV.3]{DiPernaLions89}), it holds that
\begin{equation}\label{Lpbound}
\|\theta^k\|_{L^{\infty}_tL^{\tilde p}_x} \le \|\theta_0^k\|_{L^{\tilde p}_x} + \tilde p\|g^k\|_{L^1_tL^{\tilde p}_x},
\end{equation}
and thus, by the convergence assumption on the forcing term and the initial data, the right-hand side is bounded uniformly in $k$. In particular, there exists a function $\tilde\theta\in L^{\infty}_t L^{\tilde p}_x$ and a subsequence (still denoted $\theta^k$) such that $\theta^k\rightharpoonup\tilde\theta$ weakly* in $L^\infty_tL^{\tilde p}_x$. Moreover,
because of
\begin{equation*}
\partial_t\theta^k=-\div(\theta^ku^k)+g^k+\nu^k\Delta\theta^k,
\end{equation*}
the sequence $\partial_t\theta^k$ is precompact in $L^1((0,T);H^{-s}(\T^2))$ for some $s>0$; indeed, since $u^k\to u$ strongly in $L^1_tL^p_x$ and $\theta^k\rightharpoonup\tilde \theta$ weakly* in $L^\infty_tL_x^{\tilde p}$, we have that $\theta^ku^k$ converges to $\tilde \theta u$ weakly in $L^1_tL^1_x$.

But from this precompactness and the already established weak* convergence, it follows that the convergence of $\theta^k$ towards $\tilde \theta$ takes place  in $C([0,T];L^{\tilde p}_w(\T^2))$, and thus, in particular, $\tilde \theta \in C([0,T];L^{\tilde p}_w(\T^2))$, where $L^{\tilde p}_w$ denotes the space $L^{\tilde p}$ equipped with the weak topology.

\medskip

\emph{Step 2. It holds that $\tilde \theta = \theta$. In particular, the full sequence is convergent.} Let $\varphi\in C_c^2([0,T)\times\T^2)$, then for each $k$ we have the weak formulation
\begin{equation*}
\int_0^T\int_{\T^2}\left(\partial_t\varphi\theta^k+\theta^k\nabla\varphi\cdot u^k+\nu^k \theta^k\Delta\varphi+g^k\varphi\right) \dx\dt +\int_{\T^2}\varphi(0,x)\theta^k_0\dx=0.
\end{equation*}
Owing to the convergence $\theta^k\to\tilde\theta$ in $C([0,T];L^{\tilde p}_w(\T^2))$ and the convergences of the initial data and the force, we conclude
\begin{equation*}
\int_0^T\int_{\T^2}\left(\partial_t\varphi\tilde{\theta}+\tilde{\theta}\nabla\varphi\cdot u+g\varphi\right) \dx\dt +\int_{\T^2}\varphi(0,x)\theta_0\dx=0,
\end{equation*}
so that $\tilde\theta$ is a distributional solution of the linear advection equation with data $\theta_0$. But since ${\tilde p}\geq p'$, by DiPerna-Lions theory~\cite{DiPernaLions89} we conclude that $\tilde\theta=\theta$ is the unique renormalized solution, and hence the full sequence $\theta^k$ converges to $\theta$ in $C([0,T];L^{\tilde p}_w(\T^2))$. 

\medskip

\emph{Step 3. Strong convergence {pointwise} in time.} Next, we provide the argument for strong convergence in $L^{\tilde p}(\T^2)$, pointwise in time. Since we already know convergence in $C([0,T];L_w^{\tilde p}(\T^2))$, it suffices to show convergence of the $L^{\tilde p}$-norms, {pointwise} in time.

Since ${\tilde p}\geq 2$, by Jensen's inequality all convergences remain true \emph{a fortiori} in $L^2_x$. Then, since $\theta^k$ and $\theta$ are renormalized solutions of the respective equations they satisfy, we have the estimate
\begin{equation*}
\int_{\T^2}|\theta^k(t)|^2\dx \leq \int_{\T^2}|\theta^k_0|^2\dx+2\int_0^t\int_{\T^2}g^k\theta^k \dx\dd s  
\end{equation*}
as well as the identity
\begin{equation}\label{4x}
\int_{\T^2}|\theta(t)|^2\dx = \int_{\T^2}|\theta_0|^2\dx+2\int_0^t\int_{\T^2}g\theta \dx\dd s,
\end{equation}
and consequently
\begin{equation*}
\begin{aligned}
\int_{\T^2}|\theta^k(t)|^2\dx&-\int_{\T^2}|\theta(t)|^2\dx\\
&\leq \int_{\T^2}|\theta^k_0|^2\dx-\int_{\T^2}|\theta_0|^2\dx+2\int_0^t\int_{\T^2}(g^k\theta^k-g\theta) \dx\dd s.\end{aligned}
\end{equation*}
In view of the convergence assumption on the initial data and by virtue of the weak convergence of $\theta^k$ to $\theta$, uniformly in time, and the strong convergence of $g^k$ to $g$, the right-hand side tends to zero uniformly in time. {Indeed, for $\eps$ given and any $t$, we define $K_{\eps}(t)$ as the subset of $\T^2$ in which $g$ is essentially bounded by $\eps^{-1}\|g\|_{L^1((0,T)\times\T^2)}$, and we set $K_{\eps} = \{(t,x):\: x\in K_{\eps}(t)\}$. Then $|K_{\eps}^c|\le \eps$ by construction. We then estimate
\[
\left| \int_0^t\int_{\T^2}g(\theta^k-\theta) \dx\dd s \right| \le \int_0^T \left| \int_{K_{\eps}(t)} g (\theta^k-\theta)\dx\right|\dt + \|\chi_{K^c_{\eps}}g\|_{L^1_tL^{\tilde p'}_x} \|\theta^k-\theta\|_{L^{\infty}_tL_x^{\tilde p}},
\]
and observe that for every fixed $\eps$, the first term on the right-hand side converges to zero by the dominated convergence theorem as $k\to \infty$, while the second one vanishes as $\eps\to0$ uniformly in $k$.
It follows that 
\begin{equation}\label{xx}
\limsup_{k\to\infty}\sup_t\left( \|\theta^k(t)\|_{L^2_x}- \|\theta(t)\|_{L^2_x}\right) \le 0.
\end{equation}
}

On the other hand, we have the well-known weak lower semicontinuity of the $L^2$ norm, more precisely
\begin{align*}
\MoveEqLeft\liminf_{k\to\infty}\int_{\T^2}(|\theta^k(t)|^2-|\theta(t)|^2)\dx\\
&=\liminf_{k\to\infty}\left(\int_{\T^2}|\theta^k(t)-\theta(t)|^2\dx+2\int_{\T^2}\theta(t)(\theta^k(t)-\theta(t))\dx\right)\\
&\geq \lim_{k\to\infty} 2\int_{\T^2}\theta(t)(\theta^k(t)-\theta(t))\dx=0, 
\end{align*}
for any fixed $t$, {because $\theta(t)\in L^2_x$}. Putting both inequalities together, we obtain $\theta^k(t)\to\theta(t)$ strongly in $ L^2(\T^2)$.

For ${\tilde p}>2$, interpolation between Lebesgue spaces yields $\theta^k(t)\to\theta(t)$ in $ L^q(\T^2)$ for any $q<{\tilde p}$.


\medskip

\emph{{Step 4. Strong convergence uniformly in time.}} It remains to show that the convergence takes place  in $C([0,T];L^{\tilde p}(\T^2))$. As in the previous step, it is enough to treat the case $\tilde p=2$.

Towards a contradiction, we assume that there exists a positive $\delta$,  a subsequence of $\{\theta^k\}$ that we do not relabel for notational convenience, and a sequence of points $t^k\in [0,T]$ such that
\begin{equation}\label{xxx}
\|\theta^k(t^k) - \theta(t^k)\|_{L^2_x}\ge \delta.
\end{equation}
We may extract a further subsequence (again not relabelled) such that $t^k\to t^*$ for some $t^*\in [0,T]$. By virtue of \eqref{4x} and the dominated convergence theorem, we have that
$
\|\theta(t^*)\|_{L^2_x}  = \lim_{k\to \infty} \|\theta(t^k)\|_{L^2_x}
$.
Moreover, because $\theta \in C((0,T);H^{-s}(\T^2))$ for some $s$, an approximation argument reveals that $\theta(t^k)\rightharpoonup \theta(t^*)$ weakly in $L^2_x$. It follows that
\begin{equation}\label{5x}
\lim_{k\to\infty} \|\theta(t^k)-\theta(t^*)\|_{L^2_x}=0.
\end{equation}
In view of the convergence of $\theta^k$ in $C([0,T]; L^2_w(\T^2))$, we then deduce $\theta^k(t^k) \rightharpoonup \theta(t^*)$ weakly in $L^2$. Lower semicontinuity, \eqref{xx} and \eqref{5x} then entail 
\begin{align*}
\|\theta(t^*)\|_{L^2_x}^2 & \le \liminf_{k\to \infty} \|\theta^k(t^k)\|_{L^2_x}^2 \\
&\le \limsup_{k\to \infty} \sup_t \left(\|\theta^k(t)\|_{L^2_x}^2 - \|\theta(t)\|_{L^2_x}^2\right) +\lim_{k\to \infty} \|\theta(t^k)\|_{L^2}^2\\
&\le \|\theta(t^*)\|_{L^2_x}^2.
\end{align*}
Thus 
\[
\lim_{k\to \infty}\|\theta^k(t^k) - \theta(t^*)\|_{L^2_x} = 0.
\]
The latter and \eqref{5x} and an application of the triangle inequality lead to a contradiction with \eqref{xxx}.
\end{proof}

We now turn to the proof of Theorem \ref{T1}.

In the following, we let $\ok = \omega^{\nu^k}$ and $\omega$ be the solutions to the Navier--Stokes and Euler equations considered in Theorem \ref{T1}, and write $\uk= u^{\nu^k}$, $\ok_0 = \omega^{\nu^k}_0$, and $g^k=g^{\nu^k}$, accordingly.

Inspired by  the strategy developed in \cite{ConstantinDrivasElgindi19}, we introduce two \emph{linear} problems in which $\varphi_{\ell}$ denotes a standard mollifier such that
\[
\varphi_{\ell}(x) = \ell^2 \varphi_1(\ell x).
\]
We denote by $\omega_{\ell}$   a  solution  of the linear advection equation
\begin{align}
\partial_t \omega_{\ell} + u\cdot \grad \omega_{\ell}  &= g\ast \varphi_{\ell} \label{10},\\
\omega_{\ell}(0)& = \omega_0\ast \varphi_{\ell},\label{11}
\end{align}
and by $\ok_{\ell}$ a  solution of the linear advection-diffusion equation
\begin{align}
\partial_t \ok_{\ell} + \uk \cdot \grad \ok_{\ell} & = \nu^k \laplace \ok_{\ell} + g^k\ast \varphi_{\ell},\label{12}\\
\ok_{\ell}(0) & = \ok_0\ast \varphi_{\ell}.\label{13}
\end{align}
Both problems are well-posed in the class $L^{\infty}((0,T)\times\T^2)$. Indeed, as a consequence of the mollification, the forcing terms in \eqref{10} and \eqref{12} belong to the class $L^1((0,T);L^{\infty}(\T^2))$ and the initial data in \eqref{11} and \eqref{13} belong to $L^{\infty}(\T^2)$. Hence, existence, uniqueness and stability of distributional solutions is guaranteed as a result of the DiPerna--Lions theory \cite{DiPernaLions89} provided that $u\in  L^1((0,T);W^{1,1}(\T^2))$. If $p>1$, this is certainly satisfied because in this case, our solutions to the Euler and Navier--Stokes vorticity equations belong to the class $L^{\infty}((0,T);L^p(\T^2))$, and thus, $u, u^k\in L^{\infty}((0,T);W^{1,p}(\T^2))$ as a consequence of the Calder\'on--Zygmund theory. If $p=1$, the Euler velocity field is merely a singular integral of an $L^1$ function and this line of argumentation is not applicable. Instead, uniqueness of distributional solutions to the advection equation \eqref{11} has been established in \cite{CrippaNobiliSeisSpirito17} based on a quantitative approach to continuity equations developed in \cite{Seis17}. Despite the regularizing effect of the diffusion, it is not apparent how to extend the results from \cite{CrippaNobiliSeisSpirito17,Seis17} to the diffusive setting \eqref{13} in the case $p=1$.

It is not difficult to observe that these linear problems approximate the nonlinear ones.

\begin{lemma}\label{L1}
It holds that $\omega_{\ell} \to \omega$ and $\ok_{\ell}\to \ok$ as $\ell\to \infty$ strongly in $C_tL^p_x$, uniformly in $k$.
\end{lemma}


\begin{proof}
The arguments for both statements are quite similar and straightforward. We start with the inviscid case, in which the solutions have a lower regularity.

We first notice that 
  $\eta = \omega_{\ell}-\omega$ is a difference of renormalized solutions which satisfies the advection equation
 \begin{align*}
  \partial_t \eta + u\cdot \grad \eta & = g\ast \varphi_{\ell} - g,\\
\eta(0) & = \omega_0\ast \varphi_{\ell} - \omega_0,
\end{align*}
in the sense of a renormalized solution. Note it is not an immediate consequence of the concept of renormalized solutions that differences of such solutions are themselves renormalized. That this holds true can be verified by an analysis very similar to that in Lemma 4.2 of \cite{ColomboCrippaSpirito15}, in which a damping term instead of a forcing term is considered. 
We then deduce that $\eta$ satisfies an a priori estimate analogous to \eqref{4}, thus
\[
\|\omega_{\ell} - \omega\|_{L^{\infty}_tL^p_x} \le \|g\ast \varphi_{\ell}-g\|_{L^1_tL^p_x} + \|\omega_0\ast\varphi_{\ell} - \omega_0\|_{L^p_x}.
\]
The first statement now follows from the obvious convergence of the right-hand side.

In the diffusive setting, we notice that the solutions to the Navier--Stokes equations obey the energy estimate
\[
\|\omega^k\|_{L^{\infty}_tL^p_x} + \nu^k \|\grad |\omega^k|^{\frac{p}2}\|^{\frac2p}_{L^2_tL^2_x} \lesssim \|g^k\|_{L^1_tL^p_x} + \|\omega^k_0\|_{L^p_x},
\]
and thus, $\omega^k\in L^2_tL^{q}_x$ by Sobolev embedding, for any $q\in[1,\infty)$. Via Calder\'on--Zygmund estimates, we deduce that $u$ belongs to the class $L^2_tW^{1,q}_x$ for every $q\in[1,\infty)$, and thus, by DiPerna--Lions theory, every distributional solution is automatically renormalized, and so is $\eta^k = \omega_{\ell}^k - \omega^k$, which is a difference of distributional solutions. From this observation, the estimate follows as in the inviscid case.
\end{proof}

The statement in Theorem \ref{T1} now follows via the  
 triangle inequality,
\[
\|\ok - \omega\|_{L^{\infty}_tL^p_x} \le \|\ok - \ok_{\ell}\|_{L^{\infty}_tL^p_x} + \|\ok_{\ell} - \omega_{\ell}\|_{L^{\infty}_tL^p_x} + \|\omega_{\ell}-\omega\|_{L^{\infty}_tL^p_x} ,
\]
from the  previous lemma and  Proposition \ref{P1}, the latter being applied with any $\tilde p>\max\{p,p'\}$. Notice that by construction, for any $\ell$ fixed, $\omega^k_0\ast\varphi_{\ell}\to \omega_0\ast\varphi_{\ell}$ in $L^{\tilde p}$ and $g^k\ast\varphi_{\ell}\to g\ast\varphi_{\ell}$ in $L^1_tL^{\tilde p}_x$ with estimates dependent on $\tilde p$.

\section*{Acknowledgment} The authors thank Gennaro Ciampa, Gianluca Crippa and Stefano Spirito for sharing their preliminary work \cite{CiampaCrippaSpirito20} with us. This work is partially funded by the Deutsche Forschungsgemeinschaft (DFG, German Research Foundation) under Germany's Excellence Strategy EXC 2044 --390685587, Mathematics M\"unster: Dynamics--Geometry--Structure.
\bibliography{euler}
\bibliographystyle{acm}

\end{document}